\documentclass[reqno]{amsart}
\usepackage{amsfonts}
\usepackage{graphicx}
\usepackage{amsfonts,amsmath, amssymb}
\usepackage{hyperref}
\usepackage{marginnote}
\usepackage{color}
\usepackage{mathtools}
\usepackage{cancel}
\usepackage{cite}
\usepackage{accents}

\numberwithin{equation}{section}

\newtheorem{theorem}{Theorem}[section]

\newtheorem{lemma}[theorem]{Lemma}
\newtheorem{proposition}[theorem]{Proposition}
\newtheorem{remark}[theorem]{Remark}

\def\f{\frac}

\def\b{\bar}

\usepackage{tikz}

\newcommand{\dd}{{\rm d}}

\newcommand{\Fi}{\mathbf{1}}

\newcommand{\CA}{\mathcal{A}}

\newcommand{\CF}{\mathcal{F}}
\newcommand{\CH}{\mathcal{H}}

\newcommand{\CL}{\mathcal{L}}

\newcommand{\CM}{\mathcal{M}}

\newcommand{\CT}{\mathcal{T}}
\newcommand{\CR}{\mathcal{R}}

\newcommand{\pa}{\partial}

\newcommand{\vep}{\varepsilon}

\begin{document}

\title[Instability of Vlasov-Fokker-Planck system]{Instability in a Vlasov-Fokker-Planck Binary Mixture}

\author[Z. Zhang]{Zhu Zhang}
\address[ZZ]{Department of Mathematics, The Chinese University of Hong Kong,
Shatin, Hong Kong, P.R.~China}
\email{zzhang@math.cuhk.edu.hk}

\begin{abstract}
This paper is concerned with a kinetic model of a Vlasov-Fokker-Planck system used to describe the evolution of two species of particles interacting through a potential and a thermal reservoir at given temperature. We prove that at low temperature, the homogeneous equilibrium is dynamically unstable under certain perturbations. Our work is motivated by a problem arising in \cite{EGM1}.
\end{abstract}

\date{\today}
\maketitle

\tableofcontents

\thispagestyle{empty}
\section{Introduction}
In this paper, we consider the following Vlasov-Fokker-Planck equation
\begin{equation}\label{1.1}
\left\{
\begin{aligned}
&\pa_tf_1+v_1\pa_xf_1+F(f_2)\pa_{v_1}f_1=Q(f_1),\\
&\pa_tf_2+v_1\pa_xf_2+F(f_1)\pa_{v_1}f_2=Q(f_2),
\end{aligned}\right.
\end{equation}
which is used to describe the evolution of two species of particles interacting through a potential and a reservoir at given temperature. The unknown $f_1=f_1(t,x,v)$ is the probability density of species $1$ and $f_2$ is that of species 2, which have position $x\in \Omega=[-L,L]$ or $\mathbb{R}$ and velocity $v=(v_1,v_2,v_3)\in \mathbb{R}^3$ at time $t>0$. The collisions are governed by the Fokker-Planck operator $Q(f_i):=\nabla_v\cdot\left(\mu_{\beta}\nabla_v\left(\f{f_i}{\mu_{\beta}}\right)\right)$ where $\mu_{\beta}$ is the Maxwellian $$\mu_{\beta}(v)=\left(\f{\beta}{2\pi}\right)^{\f32}e^{-\f{\beta|v|^2}{2}}$$
with the inverse temperature $\beta>0$. The evolution of the particles is influenced by a self-consistent, repulsive Vlasov force
\begin{align}
F(h)(t,x)=-\pa_x\int_{\Omega}U(|x-y|)\dd y\int_{\mathbb{R}^3}h(t,y,v)\dd v.\nonumber
\end{align}
Here the potential $U(r)$ is nonnegative, smooth, bounded with $U(r)=0$ for $|r|\geq 1$ and normalized as
$$
\int_0^\infty U(r)\dd r=1.
$$
We refer to \cite{EGM2,MM1,MM2} for more physical applications of this model.\\

The system \eqref{1.1} admits the following Lyapunov functional
\begin{align}
H(f_1,f_2):=&\int_{\Omega\times \mathbb{R}^3}\left[f_1\ln f_1+f_2\ln f_2\right]\dd x\dd v+\int_{\Omega\times \mathbb{R}^3}\f{\beta|v|^2(f_1+f_2)}{2}\dd x\dd v\nonumber\\
&+\beta\int_{\Omega\times \Omega}U(x-x')\rho_{f_1}(x)\rho_{f_2}(x')\dd x\dd x',\nonumber
\end{align}
where $\rho_{f_i}=\int_{\mathbb{R}^3}f_i(v)\dd v,$ is the local density of the $i$-species. It is straightforward to verify the following dissipation identity:
\begin{align}\label{1.3}
\f{\dd H(f_1,f_2)}{\dd t}+\sum_{i=1,2}\int_{\Omega\times \mathbb{R}^3}\f{\mu^{2}_{\beta}}{f_i}\left|\nabla_v\f{f_i}{\mu_{\beta}}\right|=0
\end{align}
along the dynamics of \eqref{1.1}. One can see from \eqref{1.3} that the temporal derivative $\f{\dd }{\dd t}H(f_1,f_2)$ vanishes when $f_i$ takes the form of a local Maxwellian, say, $f_i=\rho_{i}\mu_{\beta}$, where the density profile $\rho_i(x)$ satisfies the following equation:
 \begin{align}\label{1.6}
\ln \rho_i+\beta\int_{\Omega}U(x-x')\rho_{i+1}(x')\dd x\equiv\text{constant}.
\end{align}
Here and in the sequel, we have used the convention that $i+1=2$ if $i=1$ and $1$ if $i=2$. When $\Omega$ takes the finite interval  $\CT_L:=(-L,L),$ \eqref{1.6} is also the Euler-Lagrange system of the minimizing problem of the following free energy functional
\begin{align}
\CF(\rho_{1},\rho_{2})=&\int_{\CT_L}(\rho_{1}\ln\rho_{1}+\rho_{2}\ln\rho_{2})\dd x+\f{3}{2}\ln\f{\beta}{2\pi}\int_{\CT_L}(\rho_{1}+\rho_{2})\dd x\nonumber\\
&+\beta\int_{\CT_L\times \CT_L}U(x-x')\rho_{1}(x)\rho_{2}(x')\dd x\dd x',\nonumber
\end{align}
under the constrains of total mass
\begin{align}
\f{1}{2L}\int_{\CT_L}\rho_i\dd x=n_i,\quad i=1,2.\nonumber
\end{align}
\\

As mentioned in \cite{EGM2}, the model \eqref{1.1} undergoes a phase transition. For the value of the parameter $n_1=n_2=1$, $\beta<1$, it is proved in \cite{CCELM} that the only solution to \eqref{1.6} is the constant $[\rho_1,\rho_2]=[1,1]$, while for $\beta>1$, \eqref{1.6} admits non-constant solutions $[\rho_1(x),\rho_2(x)]$. One can expect that these minimizers are related to the stable solutions of \eqref{1.1}. Indeed, it is proved by Esposito, Guo and Marra in \cite{EGM2} that, on one hand, for $\beta<1$, the corresponding homogenous Maxwellians $[\mu_{\beta},\mu_{\beta}]$ (mixed phase) are dynamically stable. On the other hand, for $\beta>1$, the stationary solutions $[\rho_1(x)\mu_{\beta},\rho_2(x)\mu_{\beta}]$ (front) are also stable under certain perturbations. However, in \cite{EGM1} and \cite{EGM2}, it remains open, whether the homogeneous equilibrium is dynamically unstable with respect to the evolution \eqref{1.1}, for $\beta>1$. This paper is devoted to solving this problem. More precisely, our main result is the following:
\begin{theorem}\label{thm1.1}
Assume $\beta>1$. Let $N\geq 1$ be an integer and $w=(1+|v|^2)^{\theta/2}$ for $\theta\geq 0$. There exist positive constants $P_0$, $\delta_0$, $\tilde{C}_1$, $\tilde{C}_2$, and a family of spacial $P_0$-periodic solutions
$f_i^{\vep}(t,x,v)=\mu_{\beta}(v)+\sqrt{\mu_{\beta}(v)}g_i^{\vep}(t,x,v)$ to \eqref{1.1}, defined for $\vep>0$ sufficiently small, such that initially
$$f_{i}^{\vep}(0)=\mu_{\beta}+\sqrt{\mu_{\beta}}g^{\vep}_i(0)\geq 0, \quad\|w\mathbf{g}^{\vep}(0)\|_{H^N_{x,v}}\leq \tilde{C}_1\vep,$$
and
\begin{align}\label{S}
g_1^{\vep}(0,x,v_1,v_2,v_3)=g_2^{\vep}(0,-x,-v_1,v_2,v_3),
\end{align}
but
$$\sup_{0\leq t\leq T^{\vep}}\|\mathbf{g}^{\vep}(t)\|_{L^2}\geq \tilde{C}_2\delta_0.$$
Here the escape time is
$$T^{\vep}=\f{1}{Re\lambda_1}\ln\f{\delta_0}{\vep},
$$
where $\lambda_1$ is the eigenvalue with the largest real part for the linearized Vlasov-Fokker-Planck system constructed in Lemma \ref{lm3.0} with $Re\lambda_1>0$.
\end{theorem}
\begin{remark}
As mentioned above, it is proved in \cite{EGM2} that the homogeneous equilibrium is stable for $\beta<1$. On the other hand, Theorem \ref{thm1.1} implies that, for $\beta>1$, such a homogeneous equilibrium is no longer stable, even for the spacial periodic perturbations. This somehow justifies the existence of a phase transition at low temperature in this model from the mathematical point of view.
\end{remark}
\begin{remark}
The stability of the front solution for $\beta>1$ was prove by Esposito, Guo and Marra \cite{EGM2}, subject to the symmetry \eqref{S}. And so from our results one can see that the absence of such a symmetry is not an essential reason for the loss of stability.
\end{remark}
In what follows we state the idea for proving Theorem \ref{thm1.1}. Our strategy is to prove that there are two levels of instability appearing in our model. The first level is a linearized system around the global equilibrium $[\mu_{\beta},\mu_{\beta}]$. Compared with those in collisionless model \cite{GS,GS1}, the linear instability issue in collisional model could be more complicated due to dampening effects by collisions. In \cite{EGM1,EGM3}, the authors are able to adopt a new perturbation approach to prove the existence of instability in a Vlasov-Boltzmann system. Their main idea is to regard the Boltzmann operator as a bounded perturbation of the Vlasov system. However, as pointed out in \cite{EGM1}, it seems to be difficult to extend their idea into the present model, due to the unboundedness of the Fokker-Planck operator. In the present paper, we aim to directly construct the spacial periodic growing mode with sufficiently large period. This reduces to solving the eigenvalue problem \eqref{2.4} at low frequency. We formally expand the eigenvalue $\lambda(k)$ and eigenfunction $q(k)$ as a power series of frequency $k$. Such a formal expansion shows that $\f{\beta-1}{\beta}k^2$ should be an approximation to $\lambda(k)$ up to the order of $k^2$. To justify such an approximation, it suffices to solve the remainder system \eqref{2.12} and \eqref{2.14} for $[q_R,\lambda_R]$. The equation of $\lambda_R$, which is given in terms of the solvable condition of $q_R$, can be also understood as an implicit dispersion relation. From this point of view, the remainder system is indeed nonlinear. We solve it via a suitable iteration \eqref{2.15}. So that we can solve $q_R^n$ in the subspace orthogonal to the kernel of the Fokker-Planck operator at each step. The strong collision effect by the Fokker-Planck operator and the smallness of $|k|$ are essential to close the estimate.\\

To show that the instability will persist at the nonlinear level, we adopt the method of \cite{EGM1}. The key step are stated as follows. First to find the eigenvalue with the largest real part which controls the sharp growth rate of the solution. This can be done by regarding the Vlasov term as a perturbation of the kinetic Fokker-Planck operator and using Vidav's theory \cite{V}. Then we need to establish some smoothness estimates on the eigenvector of such a principle eigenvalue. Among these estimates, the pointwise estimates like \eqref{R} are crucial for further construction of the non-negative, unstable initial datum. In the Boltzmann case, such a pointwise estimate can be obtained by analysis of the characteristics due to its hyperbolic feature. While in the Fokker-Planck case, we are able to get an upper bound on higher-order norms and hence the pointwise estimates are the consequence of the classical Sobolev embedding. At the last step we show that an exponential growth estimate on the difference of a solution from equilibrium implies an exponential growth with the same rate on the first derivatives. These steps finally conclude the nonlinear instability.\\

At last, it may be interesting to study the critical case that $\beta=1$. In fact, if formally expanding the eigenvalue $\lambda(k)$ of the linear problem \eqref{2.4} up to $k^4$, one can find a stable approximate eigenvalue $\lambda(k)\approx-|k|^4\int_{\mathbb{R}}U(x)x^2\dd x$. This may gives a little clue that the mixed-phases could be stable for such a critical case. However, the mathematical justification (or disproval) would be quite challenging and we leave it for future research.\\

\noindent{\bf Notations.}  Throughout this paper, $C$ denotes a generic positive constant which may vary from line to line.  $C_a,C_b,\cdots$ denote the generic positive constants depending on $a,~b,\cdots$, respectively, which also may vary from line to line. We also use the bold symbol to denote vectors in $\mathbb{R}^2$. For instance, $\mathbf{g}$ represents the vector $(g_1,g_2)$. We denote $L^2:=L^2(\Omega\times \mathbb{R}^3)$ and its norm is denoted by $\|\cdot\|_{L^2}$. The notation $\langle\text{ },\text{ } \rangle$ represents the standard $L^2(\Omega\times \mathbb{R}^3)$ inner product. If $\mathbf{g}$ is an vector-valued function, for simplicity of the notation, we also denote $\|\mathbf{g}\|_{L^2}$ as its standard norm in $L^2(\Omega\times \mathbf{R}^3)\times L^2(\Omega\times \mathbf{R}^3)$. We define the differential operator $\pa_{\gamma}^\alpha=\pa^{\alpha}_x\pa_{v_1}^{\gamma_1}\pa_{v_2}^{\gamma_2}\pa_{v_3}^{\gamma_3},$ where $\alpha$ is related to the spacial variables, while $\gamma=[\gamma_1,\gamma_2,\gamma_3]$ is related to the velocity variables.
\section{Existence of a growing mode}
In what follows, we denote $\mu=\mu_{\beta}=\f{\beta^{3/2}}{(2\pi)^{3/2}}e^{-\f{\beta|v|^2}{2}}$ for simplicity of the notation.
Consider the following linearized Vlasov-Fokker-Planck equation around $\mu$:
\begin{equation}\label{2.1}\left\{
\begin{aligned}
&\pa_tg_1+v_1\pa_xg_1-\beta F(\sqrt{\mu}g_2)v_1\sqrt{\mu}-Lg_1=0,\\
&\pa_tg_2+v_1\pa_xg_2-\beta F(\sqrt{\mu}g_1)v_1\sqrt{\mu}-Lg_2=0.
\end{aligned}\right.
\end{equation}
Here $$Lq=\f{1}{\sqrt{\mu}}\nabla_v\cdot\left(\mu\nabla_v\left(\f{q}{\sqrt{\mu}}\right)\right).$$ For simplicity, we also write \eqref{2.1} in the vector form
\begin{align}\label{2.1-1}
\pa_t\mathbf{g}+\mathcal{L}\mathbf{g}=0.
\end{align}
In this section, we construct the growing mode to \eqref{2.1-1} for $\beta>1$. Before doing that, we define some functional spaces which will be used in this section. Define $\mathcal{H}(\mathbb{R}^3):=L^{2}(\mathbb{R}^3),$ equipped with the inner product
\begin{align}
\langle f ,g\rangle_{\mathcal{H}(\mathbb{R}^3)}=\int_{\mathbb{R}^3}f\b{g}\dd v.\nonumber
\end{align}
We denote $\nu(v):=1+|v|^2$ as the collision frequency and define $$\mathcal{H}_{\nu}(\mathbb{R}^3):=\{f\in \CH(\mathbb{R}^3), |\nu^{1/2} f|_{\CH(\mathbb{R}^3)}<\infty\}.$$ We also define $\mathcal{H}^1(\mathbb{R}^3):=\{f\in \mathcal{H}(\mathbb{R}^3),\nabla f\in \mathcal{H}(\mathbb{R}^3)\}.$ First we recall some basic properties of the Fokker-Planck operator $L$ (for instance, see \cite{AGGMMMS,DFT,HJ}), stated in the following:
\begin{lemma}\label{lm2.0}
1. $-L$ is self-adjoint in $\CH$, $Ker(-L)=\text{span}\{{\sqrt{\mu}}\}$ and $-L$ has a spectral gap on $[ker(-L)]^{\perp}$.  \\

2. $-L(v_i\sqrt{\mu})=v_i\sqrt{\mu},\quad i=1,2,3.$\\

3. Let $P$ be the projection onto $\sqrt{\mu}$
\begin{align}\label{P}
Pf:=\langle f,\sqrt{\mu}\rangle_{\CH(\mathbb{R}^3)}\sqrt{\mu}.
\end{align}
Then there exists a positive constant $c_1>0$, such that
\begin{align}\label{2.2}
\langle -Lf,f\rangle_{\CH(\mathbb{R}^3)}\geq c_1\{|\nu^{1/2}(I-P)f|_{\CH(\mathbb{R}^3)}^2+|\nabla_v(I-P)f|_{\CH(\mathbb{R}^3)}^2\}.
\end{align}
\end{lemma}

Now we turn to constructing the linear growing modes $\mathbf{g}=(g_1,g_2)$ which satisfy $\CL\mathbf{g}=-\lambda\mathbf{g}$, with $\lambda>0$ and the following symmetry condition:
\begin{align}
g_1(x,v_1,\xi)=g_2(-x,-v_1,\xi),\quad \xi=(v_2,v_3).\nonumber
\end{align}
To do this, we first seek a linear growing mode of the form $g_1=g_2$ where the function $g_1=e^{ikx}q(v_1,\xi)$ is periodic in $x$. Then the linear system \eqref{2.1-1} reduces to the following eigenvalue problem
\begin{align}\label{2.4}
(\lambda+ikv_1)q-ik\beta\hat{U}(k)(\int_{\mathbb{R}^3}q\sqrt{\mu}\dd v)v_1\sqrt{\mu}=Lq.
\end{align}
Here we have denoted that $\hat{U}(k):=\int_{\mathbb{R}}U(y)e^{-ik y}\dd y.$  Once \eqref{2.4} is solved, by using the rotation invariance of the Fokker-Planck operator and the symmetry of potential $U$, it can be checked directly that $e^{-ikx}q(-v_1,\xi)$ also satisfies \eqref{2.4}, with the same $\lambda$. Then let $$\mathbf{\hat{g}}(x,v)=(e^{ikx}q(v_1,\xi)+e^{-ikx}q(-v_1,\xi),e^{ikx}q(v_1,\xi)+e^{-ikx}q(-v_1,\xi)).$$
Clearly, $\mathbf{\hat{g}}$ is our desire. \\

Now we go back to \eqref{2.4}. Notice that when $k=0$, $\lambda=0$ is a single eigenvalue with eigenvector $\sqrt{\mu(v)}$. It is quite natural to seek an unstable eigenvalue at the lower frequency regime. To do this, we formally expand the eigenvalue $\lambda(k)$ and eigenfunction $q(k)$ as the following power series of $k$:
\begin{align}
\lambda(k)=\lambda_0+k \lambda_1+k^2\lambda_2+\cdots,\quad q(k)=q_0+kq_1+k^2q_2+\cdots.\nonumber
\end{align}
Then plug these into \eqref{2.4} and equate the same powers of $k$ to obtain:
\begin{align}
k^0&: \lambda_0q_0=L q_0,\label{2.5}\\
k^1&: (\lambda_1+iv_1)q_0+\lambda_0q_1-i\beta \hat{U}(0)\left(\int_{\mathbb{R}^3}q_0\sqrt{\mu}\dd v\right)v_1\sqrt{\mu}=Lq_1,\label{2.6}\\
k^2&:\lambda_2 q_0+\lambda_0q_2+(\lambda_1+iv_1)q_1-i\beta \hat{U}(0)\left(\int_{\mathbb{R}^3}q_1\sqrt{\mu}\dd v\right)v_1\sqrt{\mu}\nonumber\\
&\qquad-i\beta\hat{U}'(0)\left(\int_{\mathbb{R}^3}q_0\sqrt{\mu}\dd v_1\right)v_1\sqrt{\mu}=Lq_2,\label{2.7}\\
\cdots&.\nonumber
\end{align}
From \eqref{2.5}, we have $\lambda_0=0$ and
\begin{align}
q_0=\left(\int_{\mathbb{R}^3}q_0\sqrt{\mu}\dd v\right)\sqrt{\mu}.\label{2.8}
\end{align}
Without loss of generality, we assume that $\int_{\mathbb{R}^3}q_0\sqrt{\mu}\dd v=1$. Substituting \eqref{2.8} into \eqref{2.6}, we have
\begin{align}
Lq_1=\lambda_1\sqrt{\mu}+i(1-\beta\hat{U}(0))v_1\sqrt{\mu}=\lambda_1\sqrt{\mu}+i(1-\beta)v_1\sqrt{\mu},\nonumber
\end{align}
which implies that $\lambda_1=0$. Then using Lemma \ref{lm2.0}, we can solve, modulo a multiple of $\sqrt{\mu}$, that
\begin{align}
q_1=i(\beta-1)v_1\sqrt{\mu}.\label{2.9}
\end{align}
Then substitute \eqref{2.8} and \eqref{2.9} into \eqref{2.7} to get
\begin{align}
\left(\lambda_2-(\beta-1)v_1^2\right)\sqrt{\mu}=Lq_2.\label{2.10}
\end{align}
Here we have used the fact that $\hat{U}'(0)=0$, which is a consequence of the evenness of $U$. By projecting \eqref{2.10} onto $\sqrt{\mu}$, we obtain the following approximate dispersive relation (up to $k^2$) with respect to \eqref{2.4}:
\begin{align}
\lambda_2=\f{\beta-1}{\beta}>0\nonumber
\end{align}
for $\beta>1$. From \eqref{2.10}, we can also solve, modulo a multiple of $\sqrt{\mu}$, that
\begin{align}\label{2.11-1}
q_2=\lambda_2L^{-1}\{(1-\beta v_1^2)\sqrt{\mu}\}.
\end{align}
It is natural to take $[\f{\beta-1}{\beta}k^2, q_0+kq_1+k^2q_2]$ as the approximation of $[\lambda(k),q(k)]$ up to the second order. To justify such an approximation, we seek for the solution to \eqref{2.4} in the form:
 $$\lambda(k)=\f{\beta-1}{\beta}k^2+k^3\lambda_R(k),\quad \text{and}\quad q(k)=q_0+kq_1+k^2q_2+k^3q_R(k),$$
where $q_0$, $q_1$ and $q_2$ are given by \eqref{2.8}, \eqref{2.9} and \eqref{2.11-1} respectively. Then the equation of $q_R$ reads as
\begin{align}\label{2.12}
-Lq_R=&-\lambda_R\{q_0+kq_1+k^2q_2+k^3q_R\}-iv_1(q_2+kq_R)\nonumber\\
&-\f{\beta-1}{\beta}\{q_1+kq_2+k^2q_R\}+i\beta k^{-2}\{\hat{U}(k)-\hat{U}(0)-\hat{U}'(0)k\}v_1\sqrt{\mu},
\end{align}
supplemented with the orthogonal condition
\begin{align}\label{2.13}
\int_{\mathbb{R}^3}q_R\sqrt{\mu}\dd v=0.
\end{align}
Project \eqref{2.12} onto $\sqrt{\mu}$ to obtain the following equation of $\lambda_R(k)$:
\begin{align}\label{2.14}
\lambda_R(k)+i\int_{\mathbb{R}^3}v_1q_2\sqrt{\mu}\dd v+ik\int_{\mathbb{R}^3}v_1\sqrt{\mu}(I-P)q_R\dd v=0,
\end{align}
where the projection $P$ is defined in \eqref{P}. Next lemma gives the solvability of the remainder system \eqref{2.12}, \eqref{2.13} and \eqref{2.14}.
\begin{lemma}\label{lm1.1}
Assume $\beta>1$. There exists a positive constant $k_0$, such that for any $0<|k|\leq k_0$, the system \eqref{2.12}, \eqref{2.13}, \eqref{2.14} admits a unique solution $[q_R,\lambda_R]$. Moreover, it holds that
\begin{align}
q_R\in {\mathcal{H}}^1(\mathbb{R}^3)\cap\CH_{\nu}(\mathbb{R}^3)\text{ and } |\lambda_R(k)|\leq C.\nonumber
\end{align}
Here the constant $C>0$ is independent of $k$.
\end{lemma}
 \begin{proof}
The solution is constructed in terms of the following iteration scheme:
\begin{equation}\label{2.15}
\left\{
\begin{aligned}
-Lq_R^{n+1}=&-\lambda_R^{n}\{q_0+kq_1+k^2q_2+k^3q_R^{n}\}-iv_1(q_2+kq_R^{n})\\
&-\f{\beta-1}{\beta}\{q_1+kq_2+k^2q_R^n\}+i\beta k^{-2}\{\hat{U}(k)-\hat{U}(0)-\hat{U}'(0)k\}v_1\sqrt\mu,
\\
\lambda_R^{n+1}=&-i\int_{\mathbb{R}^3}v_1q_2\sqrt{\mu}\dd v-ik\int_{\mathbb{R}^3}v_1\sqrt{\mu}(I-P)q_{R}^{n+1}\dd v,\\
\end{aligned}\right.
\end{equation}
with the orthogonal condition $\int_{\mathbb{R}^3}q^{n+1}_R\sqrt{\mu}\dd v=0$ and initial datum
\begin{align}
&\lambda_R^0=-i\int_{\mathbb{R}^3}v_1q_2\sqrt{\mu}\dd v,\quad q^0_{R}=0.\nonumber
\end{align}
It is straightforward to verify that, for each $n\geq 0$, the R.H.S of the first equation in \eqref{2.15} falls in $\left[Ker(-L)\right]^{\perp}$. So that the spectral gap of the Fokker-Planck operator (see Lemma \ref{lm2.0}) guarantees the existence of $\{q_{R}^{n}\}_{n\geq1}$ in $\left[Ker(-L)\right]^{\perp}\subseteq\CH$. It remains to establish some uniform estimates on $q_R^n$ and $\lambda_R^n$ and their convergence.\\

\underline{Uniform estimate: } Taking the inner product of \eqref{2.15} with $q_R^{n+1}$, we get
\begin{align}\label{2.16}
&\langle-Lq_R^{n+1},q_R^{n+1}\rangle_{\CH(\mathbb{R}^3)}\nonumber\\
&\quad=\langle-\lambda_R^n\{q_0+kq_1+k^2q_2+k^3q_R^n\}, q_R^{n+1}\rangle_{\CH(\mathbb{R}^3)}\nonumber\\
&\quad\quad-\f{\beta-1}{\beta}\langle q_1+kq_2+k^2q_R^n, q_R^{n+1}\rangle_{\CH(\mathbb{R}^3)}+\langle -iv_1(q_2+kq_R^n),q_R^{n+1}\rangle_{\CH(\mathbb{R}^3)}\nonumber\\
&\quad\quad+\langle i\beta k^{-2}\{\hat{U}(k)-\hat{U}(0)-\hat{U}'(0)k\}v_1\sqrt{\mu},q_R^{n+1}\rangle_{\CH(\mathbb{R}^3)}.
\end{align}
Since $$\int_{\mathbb{R}^3}q^{n+1}_R\sqrt{\mu}\dd v=0,$$
then from the coercivity estimate \eqref{2.2}, it holds that
\begin{align}\label{2.17}
\langle-Lq_R^{n+1},q_R^{n+1}\rangle_{\mathcal{H}(\mathbb{R}^3)}\geq c_1\{|\nabla_vq_{R}^{n+1}|_{\mathcal{H}(\mathbb{R}^3)}^2+|\nu^{1/2} q_R^{n+1}|_{\mathcal{H}(\mathbb{R}^3)}^2\},
\end{align}
for some positive constants $c_1>0$. Notice that
$$
k^{-2}|\hat{U}(k)-\hat{U}(0)-k\hat{U}'(0)|\leq \f{|\hat{U}''|_{L^{\infty}}}{2}\leq C.
$$
Then by Cauchy-Schwarz, the R.H.S of \eqref{2.16} is bounded by
\begin{align}
&\f{c_1}{2}|q_R^{n+1}|_{\CH(\mathbb{R}^3)}^2+Ck^2\{|\lambda_R^n|^2+1\}|\nu^{1/2} q_R^{n}|_{\CH(\mathbb{R}^3)}^2\nonumber\\
&\quad+C\{|\lambda_R^n|^2+1\}\cdot\{|q_0|_{\CH(\mathbb{R}^3)}^2+|q_1|_{\CH(\mathbb{R}^3)}^2+|\nu^{1/2} q_2|_{\CH({\mathbb{R}^3})}^2+1\},\nonumber\\
&\leq \f{c_1}{2}|q_R^{n+1}|_{\CH(\mathbb{R}^3)}^2+Ck^2\{|\lambda_R^n|^2+1\}|\nu^{1/2} q_R^{n}|_{\CH(\mathbb{R}^3)}^2+C\{|\lambda_R^n|^2+1\},\nonumber
\end{align}
for some positive constant $C$ depending only on $c_1$. Combine this with \eqref{2.16} and \eqref{2.17} to obtain, for some positive constant $\hat{C}_0$ independent of $k$ and $n$, that
\begin{align}\label{2.18}
|\nabla_vq_{R}^{n+1}|_{\mathcal{H}(\mathbb{R}^3)}^2+|\nu^{1/2} q_R^{n+1}|_{\mathcal{H}(\mathbb{R}^3)}^2\leq \hat{C}_0k^2\{|\lambda_R^n|^2+1\}|\nu^{1/2} q_R^{n}|_{\CH(\mathbb{R}^3)}^2+\hat{C}_0\{|\lambda_R^n|^2+1\}.
\end{align}
We also have, from the second equation of \eqref{2.15}, that
\begin{align}\label{2.19}
|\lambda_{R}^{n+1}|\leq \big|\int_{\mathbb{R}^3}v_1\sqrt{\mu}q_2\dd v\big|+k\big|\int_{\mathbb{R}^3}v_1\sqrt{\mu}(I-P)q_R^{n+1}\dd v\big|\leq \hat{C}_1+\f{k}{\sqrt{\beta}}|q_R^{n+1}|_{\CH(\mathbb{R}^3)},
\end{align}
where the constant $\hat{C}_1>0$ is also independent of $k$ and $n$. Now we use an induction argument to show that for each $n\geq 0$, $q_R^{n}\in \CH^1\cap\CH_{\nu}$,
\begin{align}\label{2.20}
|\lambda_R^n|\leq 2\hat{C}_1\text{ and } |\nabla_vq_{R}^{n}|_{\mathcal{H}(\mathbb{R}^3)}^2+|\nu^{1/2} q_R^{n}|_{\mathcal{H}(\mathbb{R}^3)}^2\leq 2\hat{C}_0(4\hat{C}_1^2+1),
\end{align}
provided that $0<|k|\leq k_0$ for $k_0$ suitably small. Notice that \eqref{2.20} holds for $n=0$. Assume it holds for $n=N$, then applying \eqref{2.18} to $q^{N+1}_R$, we have
\begin{align}\label{2.21}
|\nabla_vq_{R}^{N+1}|_{\mathcal{H}(\mathbb{R}^3)}^2+|\nu^{1/2} q_R^{N+1}|_{\mathcal{H}(\mathbb{R}^3)}^2&\leq 2\hat{C}_0^2k^2(4\hat{C}_1^2+1)^2+\hat{C}_0(4\hat{C}^2_1+1)\nonumber\\
&\leq \hat{C}_0(4\hat{C}^2_1+1)\cdot\{1+2\hat{C}_0(4\hat{C}_1^2+1)k^2\}.
\end{align}
Substituting \eqref{2.21} into \eqref{2.19}, we have
\begin{align}\label{2.21-1}
|\lambda_R^{N+1}|\leq \hat{C}_1+\f{k}{\sqrt{\beta}}\left(\hat{C}_0(4\hat{C}^2_1+1)\cdot\{1+2\hat{C}_0(4\hat{C}_1^2+1)k^2\}\right)^{1/2}.
\end{align}
Let $0<k_0\leq \min\left\{\sqrt{\f{1}{2\hat{C}_0(4\hat{C}_1^2+1)}}, \hat{C}_1\sqrt{\f{\beta}{2\hat{C}_0(4\hat{C}_1^2+1)}}\right\}$. Then it is straightforward to verify, from \eqref{2.21} and \eqref{2.21-1} that \begin{align}
|\lambda_R^{N+1}|\leq 2\hat{C}_1\text{ and } |\nabla_vq_{R}^{N+1}|_{\mathcal{H}(\mathbb{R}^3)}^2+|\nu^{1/2} q_R^{N+1}|_{\mathcal{H}(\mathbb{R}^3)}^2\leq 2\hat{C}_0(4\hat{C}_1^2+1).\nonumber
\end{align}
This has justified the validity of estimate \eqref{2.20} for $N+1$. To show the convergence, consider the difference $q_{R}^{n+1}-q_{R}^n$. It is direct to see that $q_{R}^{n+1}-q_R^n$ solves
\begin{align}
-L(q_{R}^{n+1}-q_R^n)=&-(\lambda_R^n-\lambda_R^{n-1})\cdot(q_0+kq_1+k^2q_2+k^3q_{R}^n)\nonumber\\
&-k^3\lambda_R^{n-1}
(q_R^n-q_R^{n-1})-iv_1k(q_R^n-q_R^{n-1})-\f{k^2(\beta-1)}{\beta}(q_R^{n}-q_R^{n-1}).\nonumber
\end{align}
Notice that $\lambda_R^n-\lambda_R^{n-1}=-ik\int_{\mathbb{R}^3}v_1\sqrt{\mu}(I-P)(q_R^n-q_R^{n-1})\dd v$. Then using the same energy method, we get, for some constant $\hat{C}_2>0$ independent of $k$ and $n$, that
\begin{align}
|\nabla(q_{R}^{n+1}-q_R^n)|_{\CH(\mathbb{R}^3)}^2+|\nu^{1/2}(q_{R}^{n+1}-q_R^n)|_{\CH(\mathbb{R}^3)}^2\leq \hat{C}_2k|\nu^{1/2}(q_R^n-q_R^{n-1})|_{\CH(\mathbb{R}^3)}^2.\nonumber
\end{align}
We further take $0<k_0\leq \f{1}{2\hat{C}_2}$. Then for any $0<k\leq k_0$, $\{q_R^n\}_{\geq 1}$ is a Cauchy sequence in $\CH_{1}\cap\CH_{\nu}$. The solution pair $(q_R,\lambda_R)$ is obtained by passing to the limit $n\rightarrow+\infty$. The uniqueness is standard. This completes the proof of Lemma \ref{lm1.1}.
\end{proof}
From Lemma \ref{lm1.1}, we can directly obtain the following result of linear instability.
\begin{proposition}\label{prop1.1}
Assume $\beta>1$. There exists a positive constant $k_0$, such that for any $0<k\leq k_0$, the linear problem \eqref{2.4} has a positive eigenvalue $\lambda(k)$ with multiplicity 1. Moreover, the eigenvalue $\lambda(k)$ and the eigenvector $q(k)$ (normalized by $\int_{\mathbb{R}^3}q(k)\sqrt{\mu}\dd v=1$) have the following asymptotical structure:
\begin{align}
\lambda(k)=\f{\beta-1}{\beta}k^2+k^3\lambda_R(k),\nonumber
\end{align}
and
\begin{align}
q(k)=\sqrt{\mu}+ik(\beta-1)v_1\sqrt{\mu}+\f{k^2(\beta-1)}{\beta}L^{-1}\left(\{1-\beta v_1^2\}\sqrt{\mu}\right)+k^3q_{R}(k).\nonumber
\end{align}
Here $\int_{\mathbb{R}^3}q_R\sqrt{\mu}\dd v=0$ and $|q_R(k)|_{\CH^1}+|q_R(k)|_{\CH_\nu}+|\lambda_R(k)|\leq C$ for some constants $C>0$ independent of $k$.
\end{proposition}
\section{Nonlinear instability}
In this section, we will show that the linear stability indeed leads to a nonlinear stability. Define $\mathcal{H}(\Omega\times\mathbb{R}^3)=L^2(\Omega\times\mathbb{R}^3)$, equipped with the inner product
\begin{align}
\langle f,g\rangle_{\mathcal{H}(\Omega\times\mathbb{R}^3)}=\int_{\Omega}\int_{\mathbb{R}^3}f\b{g}\dd x\dd v.\nonumber
\end{align}
In what follows, $\mathcal{H}(\Omega\times\mathbb{R}^3)$ is denoted by $\mathcal{H}$ for short.
We consider the subspace
$$\CM:=\{\mathbf{g}=(g_1,g_2)\text{ }|g_1(x,v_1,\xi)=g_2(-x,-v_1,\xi)\}\subseteq \mathcal{H}\times \mathcal{H},$$
equipped with the standard inner product
$$
\langle \mathbf{f},\mathbf{g}\rangle_{\CM}=\langle f_1,g_1\rangle_{\mathcal{H}}+\langle f_2,g_2\rangle_{\mathcal{H}}.
$$
Recall the linearized VFP operator $\CL$ \eqref{2.1-1}. Next lemma gives the existence of a dominating eigenvalue of $-\CL$.
\begin{lemma}\label{lm3.0}
\cite{EGM1} Assume $\beta>1$. Then for all $\zeta>0$, the spectrum of $-\CL$ in $\{Re\lambda>\zeta\}$ consists of a finite number of eigenvalues of finite multiplicity (non-empty by Proposition \ref{prop1.1}). Let $\lambda_1$ be the eigenvalue with maximal real part. Then for any $\Lambda>Re\lambda_1$, there exists a positive constant $C_{\Lambda}$, such that
\begin{align}\label{3.0}
\|e^{-t\CL}\mathbf{g}_0\|_{\CM}\leq C_{\Lambda}e^{\Lambda t}\|\mathbf{g}_0\|_{\CM}.
\end{align}
\end{lemma}
\begin{proof}
 We split \begin{align}\label{D}
\CL\mathbf{g}=\CA \mathbf{g}+K\mathbf{g},\end{align}
where
$(\CA\mathbf{g})_i=v_1\pa_xg_i-Lg_i$ and $(K\mathbf{g})_i=F(\sqrt{\mu}g_{i+1})v_1\sqrt{\mu}$. It is straightforward to verify that $e^{-t\CA}$ is contractive in $\CM$
and that $K$ is compact. Then Lemma \ref{lm3.0} follows from Vidav's lemma \cite{V}.
\end{proof}
Next lemma gives the smoothness of the growing modes.
\begin{lemma}\label{lm3.1}
Let $\mathbf{R}=(R_1,R_2)\in \CM$ be the eigenvector of $-\CL$ with $Re\lambda>0$. Then $\mathbf{R}\in C^{\infty}$. Moreover, for any integer $N\geq 0$ and weight $w(v)=e^{\f{q\beta|v|^2}{4}}$ with $0<q\leq 1/2$, we have \begin{align}
\sum_{|\alpha|+|\gamma|\leq N}\|w\pa_{\gamma}^\alpha\mathbf{R}\|_{\CM}\leq C_{N,q}\|\mathbf{R}\|_{\CM}.\label{3.1-1}
\end{align}
\end{lemma}
\begin{proof}
Recall the decomposition $\CL=\CA+K$. By the Duhamel principle, the eigenvector $\mathbf{R}$ can be expressed by
\begin{align}
\mathbf{R}=-\int_0^{\infty} e^{-\lambda t}e^{-t\CA}K\mathbf{R}\dd t.\label{4.1-0}
\end{align}
Notice that $K\mathbf{R}\in C^{\infty}$. It suffices to consider the propagation of regularity by $e^{-t\CA}$. Let $g$ be the solution to the following Cauchy problem of the kinetic Fokker-Planck equation
\begin{align}\label{4.1-1}
\pa_t g+v_1\pa_xg-Lg=0,\quad g|_{t=0}=g_0,
\end{align}
and let $N\geq 0$ be an integer. Take $\pa_\gamma^\alpha$ ($|\alpha|+|\gamma|\leq N$) to \eqref{4.1-1} to get
\begin{align}\label{4.1-2}
\pa_t\pa_\gamma^\alpha g+v_1\pa_{\gamma}^\alpha g-L\pa_\gamma^\alpha g=-\sum_{0<\gamma_1\leq \gamma}\pa_{\gamma_1}v_1\pa_{\gamma-\gamma_1}^{\alpha}g-\sum_{0<\gamma_1\leq \gamma}\pa_{\gamma_1}\left(\f{\beta^2|v|^2}{4}\right)\pa_{\gamma-\gamma_1}^\alpha g.
\end{align}
First consider the pure $x$-derivative case ($\gamma=0$), where R.H.S of \eqref{4.1-2} vanishes. Taking the inner product of \eqref{4.1-2} with $w^2\pa^{\alpha}g$, we get
\begin{align}
\f{1}{2}\f{\dd }{\dd t}\|w\pa^\alpha g\|_{\CH}^2+\langle-L\pa^\alpha g,w^2\pa^\alpha g\rangle_{\CH}=0.\label{4.1-3}
\end{align}
Notice that $-L=-\Delta+\f{\beta^2|v|^2}{4}-\f{3\beta}{2}$. Then from integrating by parts, we have
\begin{align}
\int_{\mathbb{R}^3}w^2\overline{\pa^\alpha g}(-L\pa^\alpha g)\dd v=\int_{\mathbb{R}^3}w^2|\nabla_v \pa^\alpha g|^2\dd v+\int_{\mathbb{R}^3}\bigg\{\f{\beta^2|v|^2}{4}-\f{3\beta}{2}-\f{\nabla_v(w\nabla_vw)}{w^2}\bigg\}|w\pa^{\alpha}g|^2\dd v.\nonumber
\end{align}
A direct computation shows that
$$\f{\beta^2|v|^2}{4}-\f{3\beta}{2}-\f{\nabla_v(w\nabla_vw)}{w^2}=\f{\beta^2|v|^2}{4}-\f{3\beta}{2}-\f{\beta^2q^2|v|^2}{2}-\f{3\beta q}{2}\geq \f{\beta^2|v|^2}{8}-\f{3\beta(q+1)}{2}.
$$ Therefore, we have, from a classical interpolation, that
\begin{align}
\langle-L\pa^\alpha g, w^2\pa^\alpha g\rangle_{\CH}\geq \|w\nabla_v\pa^\alpha g\|_{\CH}^2+\f{\beta^2}{16}\|w \nu^{1/2}\pa^\alpha g\|_{\CH}^2-C_{\beta,q}\|\pa^\alpha g\|_{\CH}^2.\label{4.1-4}
\end{align}
Since $\|\pa^\alpha g\|_{\CH}=\|\pa^\alpha e^{-t\CA}g_0\|_{\CH}=\|e^{-t\CA}\pa^{\alpha}g_0\|_{\CH}=\|\pa^\alpha g_0\|_{\CH}$, then combining \eqref{4.1-4} with \eqref{4.1-3}, we obtain
\begin{align}
&\|w\pa^\alpha g(t)\|_{\mathcal{H}}^2+\int_0^t\|w\nabla_v\pa^\alpha g(s)\|_{\CH}^2+\|\nu^{1/2} w\pa^\alpha g(s)\|_{\CH}^2\dd s\nonumber\\
&\quad\leq C\|w\pa^\alpha g_0\|_{\mathcal{H}}^2+C\int_{0}^t\|\pa^\alpha g(s)\|_{\mathcal{H}}^2\dd s\leq C(t+1)\|w\pa^{\alpha }g_0\|_{\mathcal{H}}^2.\label{4.1-5}
\end{align}
Now consider the estimates involving $v$-derivatives ($\gamma\neq0$). Taking the inner product of \eqref{4.1-2} with $w^2\pa_{\gamma}^\alpha g$, we get
\begin{align}\label{4.1-6}
&\f{1}{2}\f{\dd }{\dd t}\|w\pa^\alpha_\gamma g\|_{\CH}^2+\langle-L\pa^\alpha_\gamma g, w^2\pa_{\gamma}^\alpha g\rangle_{\CH}\nonumber\\
&\quad=-\sum_{0<\gamma_1\leq \gamma}\langle\pa_{\gamma_1}v_1\pa_{\gamma-\gamma_1}^{\alpha}g,w^2\pa_\gamma^\alpha g\rangle_{\CH}-\sum_{0<\gamma_1\leq \gamma}\langle\pa_{\gamma_1}\left(\f{\beta^2|v|^2}{4}\right)\pa_{\gamma-\gamma_1}^\alpha g,w^2\pa_\gamma^\alpha g\rangle_{\CH}.
\end{align}
By Cauchy Schwarz, the R.H.S of \eqref{4.1-6} is
\begin{align}\label{4.1-7}
\leq C\sum_{|\alpha_1|+|\gamma_1|\leq N}\sum_{|\gamma_1|\leq |\gamma|-1}\{\|w\nu^{1/2}\pa^{\alpha_1}_{\gamma_1} g\|_{\CH}^2+\|w\nabla_v\pa^{\alpha_1}_{\gamma_1} g\|_{\CH}^2\}
\end{align}
Similar as \eqref{4.1-4}, we have
\begin{align}
\langle-L\pa^\alpha_\gamma g, w^2\pa^\alpha_\gamma g\rangle_{\CH}&\geq\|w\nabla_v\pa^\alpha_\gamma g\|_{\CH}^2+\f{\beta^2}{16}\|w\nu^{1/2} \pa^\alpha_\gamma g\|_{\CH}^2-C\|\pa^\alpha_\gamma g\|_{\CH}^2\nonumber\\
&\geq \|w\nabla_v\pa^\alpha_\gamma g\|_{\CH}^2+\f{\beta^2}{16}\|w \nu^{1/2}\pa^\alpha_\gamma g\|_{\CH}^2-C\sum_{|\alpha_1|+|\gamma_1|\leq N}\sum_{|\gamma_1|\leq |\gamma|-1}\|\nabla_v\pa^{\alpha_1}_{\gamma_1} g\|_{\CH}^2.\label{4.1-8}
\end{align}
Plug \eqref{4.1-7} and \eqref{4.1-8} back into \eqref{4.1-6} to get
\begin{align}
&\|w\pa^\alpha_\gamma g(t)\|_{\mathcal{H}}^2+\int_0^t\|w\nabla_v\pa^\alpha_\gamma g(s)\|_{\CH}^2+\|\nu^{1/2} w\pa^\alpha_\gamma g(s)\|_{\CH}^2\dd s\nonumber\\
&\quad\leq C\|w\pa^\alpha_\gamma g_0\|_{\mathcal{H}}^2+C\sum_{|\alpha_1|+|\gamma_1|\leq N}\sum_{|\gamma_1|\leq |\gamma|-1}\int_{0}^t\{\|\nu^{1/2} w\pa^{\alpha_1}_{\gamma_1} g(s)\|_{\mathcal{H}}^2+\| w\nabla_v\pa^{\alpha_1}_{\gamma_1} g(s)\|_{\mathcal{H}}^2\}\dd s.\label{4.1-9}
\end{align}
A suitable weighted summation of \eqref{4.1-5} and \eqref{4.1-9} over $|\alpha|+|\gamma|\leq N$ yields that
\begin{align}
\sum_{|\alpha|+|\gamma|\leq N}\|w\pa^\alpha_\gamma g(t)\|_{\mathcal{H}}^2\leq C(1+t)^N\cdot\{\sum_{|\alpha|+|\gamma|\leq N}\|w\pa^\alpha_\gamma g_0\|_{\mathcal{H}}^2\}.\label{4.1-10}
\end{align}
Denote $\|\mathbf{g}\|_{\mathcal{X}}:=\sum_{|\alpha|+|\gamma|\leq N}\|w\pa_\gamma^\alpha\mathbf{g}\|_{\CH}$.
Then from \eqref{4.1-10}, we have
\begin{align}\nonumber
\|e^{-t\CA}\mathbf{g}_0\|_{\mathcal{X}}^2\leq C(t+1)^N\|\mathbf{g}_0\|_{\mathcal{X}}^2.
\end{align}
Combine this with \eqref{4.1-0} to get
\begin{align}
\|\mathbf{R}\|_{\mathcal{X}}&\leq C\int_0^te^{-Re\lambda t}\|e^{-t\CA}K\mathbf{R}\|_{\mathcal{X}}\dd t\leq C\int_0^te^{-Re\lambda t}(1+t)^\f{N}{2}\|K\mathbf{R}\|_{\mathcal{X}}\dd t\nonumber\\
&\leq C_N \|K\mathbf{R}\|_{\mathcal{X}}\leq C\|\mathbf{R}\|_{\CH}.
\end{align}
Here we have used the smoothness of $U$ in the last inequality. This completes the proof of Lemma \ref{lm3.1}.
\end{proof}

Now consider the full nonlinear problem
\begin{equation}\label{3.2-0}\left\{
\begin{aligned}
&\pa_t g_i+v_1\pa_xg_i-\beta F(\sqrt{\mu}g_{i+1})v_1\sqrt{\mu}-Lg_i=-F(\sqrt{\mu}g_{i+1})\pa_{v_1}g_i+\f{\beta}{2}F(\sqrt{\mu}g_{i+1})v_1g_{i},\\
&\mathbf{g}(0)=\mathbf{g}_0.
\end{aligned}\right.
\end{equation}
Next lemma shows that, if there is an exponential growth estimate of the $L^2$-norm of the perturbation, then one can bound the growth of the first order derivatives by the same rate.
\begin{lemma}\label{lm3.2}
Let $T>0$, $w(v)=(1+|v|^2)^{\f{\theta}{2}}$ ($\theta\geq 0$) be the weight function and $\mathbf{g}$ be the solution to the nonlinear equation \eqref{3.2-0} over $[0,T]$. Assume that $Re\lambda>0$ and
\begin{align}\label{3.2-3.1}
\|\mathbf{g}(t)\|_{L^2}\leq Ce^{Re\lambda t}\|\mathbf{g}_0\|_{L^2}, \text{ for } 0\leq t\leq T.
\end{align}
Then there exists a constant $\eta>0$, depending only on $\lambda$, such that if
\begin{align}\label{3.2-3}
\sup_{0\leq t\leq T}\|\mathbf{g}(s)\|_{L^2}\leq \eta,
\end{align}
then it holds that
\begin{align}\label{3.2-1}
\|w\mathbf{g}(t)\|_{L^2}+\|w\pa_{x,v}\mathbf{g}(t)\|_{L^2}\leq Ce^{Re\lambda t}\{\|w\pa_{x,v}\mathbf{g}_0\|_{L^2}+\|w\mathbf{g}_0\|_{L^2}\}.
\end{align}
\end{lemma}
\begin{proof}
Taking the inner product of \eqref{3.2-0} with $g_i$, we have
\begin{align}\label{3.2-2}
\f{1}{2}\f{\dd }{\dd t}\|g_i\|_{L^2}^2+\underbrace{\langle-Lg_i,g_i\rangle}_{J_1}=&\underbrace{\langle \beta F(\sqrt{\mu}g_{i+1})v_1\sqrt{\mu},g_i\rangle}_{J_2}\nonumber\\
&+\underbrace{\langle -F(\sqrt{\mu}g_{i+1})\pa_{v_1}g_i,g_i\rangle}_{J_3}+
\underbrace{\langle \f\beta2 F(\sqrt{\mu}g_{i+1})g_i v_1,g_i\rangle}_{J_4}.
\end{align}
Firstly,
\begin{align}\label{J1}
J_1=\langle -\Delta g_i+\left(\f{\beta^2|v|^2}{4}-\f{3\beta}{2}\right)g_i,g_i\rangle\geq \|\nabla_vg_i\|_{L^2}^2+\f{\beta}{4}\|\nu^{1/2} g_i\|_{L^2}^2-2\beta\|g_i\|_{L^2}^2.
\end{align}
By Cauchy-Schwarz,
$$
|J_2|+|J_4|\leq C\|g_i\|_{L^2}^2+C\|g_{i+1}\|_{L^2}^2+C\|\nu^{1/2} g_i\|_{L^2}^2\cdot\sup_{0\leq s\leq t}\|g_{i+1}(s)\|_{L^2}.
$$
As for $J_3$, it holds that
$$
J_3=\langle-F(\sqrt{\mu}g_{i+1}),\pa_{v_1}\left(\f{g_i^2}{2}\right)\rangle=0.
$$
Substituting these estimates into \eqref{3.2-2}, we obtain
\begin{align}\label{3.2-4}
\f12\f{\dd}{\dd t}\|\mathbf{g}\|_{L^2}^2+\|\nabla_v\mathbf{g}\|_{L^2}^2+\|\nu^{1/2} \mathbf{g}\|_{L^2}^2\leq C\|\mathbf{g}\|_{L^2}^2+
C\sup_{0\leq s\leq t}\|\mathbf{g}(s)\|_{L^2}\cdot\|\nu^{1/2} \mathbf{g}\|_{L^2}^2.
\end{align}
Next to consider $x$-derivative. Take $\pa_x$ of \eqref{3.2-0} to get
\begin{align}\label{3.2-5}
&\pa_t\pa_xg_i+v_1\pa_x\pa_xg_i-L\pa_xg_i\nonumber\\
&\quad=\beta\pa_xF(\sqrt{\mu}g_{i+1})v_1\sqrt{\mu}-F(\sqrt{\mu}g_{i+1})\pa_{v_1}\pa_xg_i
\nonumber\\
&\quad\quad-\pa_xF(\sqrt{\mu}g_{i+1})\pa_{v_1}g_i+\f{\beta}{2}\pa_xF(\sqrt{\mu}g_{i+1})v_1g_i+\f{\beta}{2}F(\sqrt{\mu}g_{i+1}){v_1}\pa_xg_i.
\end{align}
Then taking the inner product of \eqref{3.2-5} with $\pa_xg_i$, we obtain
\begin{align}\label{3.2-6}
&\f{1}{2}\f{\dd }{\dd t}\|\pa_xg_i\|_{L^2}^2+\langle -L\pa_xg_i,\pa_xg_i \rangle-\langle\beta\pa_xF(\sqrt{\mu}g_{i+1})v_1\sqrt{\mu},\pa_xg_i  \rangle\nonumber\\
&\quad=\langle -F(\sqrt{\mu}g_{i+1})\pa_{v_1}\pa_xg_i,\pa_xg_i\rangle+\langle-\pa_xF(\sqrt{\mu}g_{i+1})\pa_{v_1}g_i,\pa_xg_i\rangle\nonumber\\
&\quad\quad+\f{\beta}{2}\langle\pa_xF(\sqrt{\mu}g_{i+1})v_1g_i,\pa_xg_i\rangle+\f\beta2\langle F(\sqrt{\mu}g_{i+1})v_1\pa_xg_i,\pa_xg_i\rangle.
\end{align}
By the coercivity estimate \eqref{2.2}, we have
\begin{align}
\langle -L\pa_xg_i,\pa_xg_i \rangle\geq c_1\{\|\nu^{1/2}(I-P)\pa_x g_i\|_{L^2}^2+\|\nabla_v(I-P)\pa_xg_i\|_{L^2}^2\}.\nonumber
\end{align}
Integrating by parts leads to
\begin{align}
&|\langle\beta\pa_xF(\sqrt{\mu}g_{i+1})v_1\sqrt{\mu},\pa_xg_i  \rangle|+|\langle -F(\sqrt{\mu}g_{i+1})\pa_{v_1}\pa_xg_i,\pa_xg_i\rangle|\nonumber\\
&\quad=|\langle\beta\pa_{xx}F(\sqrt{\mu}g_{i+1})v_1\sqrt{\mu},g_i\rangle|+|\langle -F(\sqrt{\mu}g_{i+1}),\f{\pa_{v_1}(\pa_xg_i)^2}{2}\rangle|\leq C\|\mathbf{g}\|_{L^2}^2.\nonumber
\end{align}
By Cauchy-Schwarz, it holds that
\begin{align}
&|\langle-\pa_xF(\sqrt{\mu}g_{i+1})\pa_{v_1}g_i,\pa_xg_i\rangle|+\f\beta2|\langle\pa_xF(\sqrt{\mu}g_{i+1})v_1g_i,\pa_xg_i\rangle|\nonumber\\
&\quad\leq C\sup_{0\leq s\leq t}\|g_{i+1}(s)\|_{L^2}\{\|\pa_{v_1} g_i\|_{L^2}^2+\|\pa_x g_i\|_{L^2}^2+\|\nu^{1/2} g_i\|_{L^2}^2\}\nonumber.
\end{align}
It remains to show that one can control the last term on the R.H.S of \eqref{3.2-6}. Recall the projection $P$ defined in \eqref{P}. We use $g_i=Pg_i+(I-P)g_i$ to split
\begin{align}
&\f\beta2\langle F(\sqrt{\mu}g_{i+1})v_1\pa_xg_i,\pa_xg_i\rangle\nonumber\\
&\quad=\f\beta2\langle F(\sqrt{\mu}g_{i+1})v_1\pa_xPg_i,\pa_xPg_i\rangle+\beta\langle F(\sqrt{\mu}g_{i+1})v_1\pa_xPg_i,\pa_x(I-P)g_i\rangle\nonumber\\
&\qquad+\f\beta2\langle F(\sqrt{\mu}g_{i+1})v_1\pa_x(I-P)g_i,\pa_x(I-P)g_i\rangle.\nonumber
\end{align}
The first term on the R.H.S vanishes due to the evenness of $\sqrt{\mu}$. Other two terms are controlled in terms of Cauchy-Schwarz. And so
\begin{align}
\f\beta2|\langle F(\sqrt{\mu}g_{i+1})v_1\pa_xg_i,\pa_xg_i\rangle|\leq C\sup_{0\leq s\leq t}\|g_{i+1}(s)\|_{L^2}\{\|\pa_xg_i\|_{L^2}^2+\|\nu^{1/2}\pa_x(I-P)g_i\|_{L^2}^2\}.\nonumber
\end{align}
Substituting these estimates into \eqref{3.2-6}, we obtain that
\begin{align}
&\f12\f{\dd}{\dd t}\|\pa_x\mathbf{g}\|_{L^2}^2+\f{c_1}{2}\{\|\nu^{1/2}(I-P)\pa_x\mathbf{g}\|_{L^2}^2+\|\nabla_v(I-P)\pa_x\mathbf{g}\|_{L^2}^2\}\nonumber\\
&\quad\leq C\|\mathbf{g}\|_{L^2}^2+C\{\|\nabla_v \mathbf{g}\|_{L^2}^2+\|\pa_x \mathbf{g}\|_{L^2}^2+\|\nu^{1/2} \mathbf{g}\|_{L^2}^2+\|\nu^{1/2}(I-P)\pa_x\mathbf{g}\|_{L^2}^2\}\cdot\sup_{0\leq s\leq t}\|\mathbf{g}(s)\|_{L^2}.\nonumber
\end{align}
Combine this with \eqref{3.2-4} and use \eqref{3.2-3} to get, for $\eta$ sufficiently small, that
\begin{align}\label{3.2-7-1}
\f12\f{\dd}{\dd t}\{\|\mathbf{g}\|^2_{L^2}+\|\pa_x\mathbf{g}\|_{L^2}^2\}\leq C\|\mathbf{g}\|_{L^2}^2+C\eta\cdot\|\pa_{x}\mathbf{g}\|_{L^2}^2.
\end{align}
Now applying Gronwall's inequality to \eqref{3.2-7-1}, we obtain, for suitably small $\eta>0$, that
\begin{align}\label{3.2-7}
&\|\pa_x\mathbf{g}\|_{L^2}^2+\|\mathbf{g}\|_{L^2}^2\nonumber\\
&\quad\leq Ce^{C\eta t}\{\|\pa_x\mathbf{g}_0\|_{L^2}^2+\|\mathbf{g}_0\|_{L^2}^2\}+C\int_0^te^{C\eta(t-s)}\|\mathbf{g}(s)\|_{L^2}^2\dd s \nonumber\\
&\quad\leq Ce^{C\eta t}\{\|\pa_x \mathbf{g}_0\|_{L^2}^2+\|\mathbf{g}_0\|_{L^2}^2\}+Ce^{2Re\lambda t}\int_0^te^{-(2Re\lambda-C\eta)(t-s)}\dd s\cdot \sup_{0\leq s\leq t}e^{-2Re\lambda s}\|\mathbf{g}(s)\|_{L^2}^2\nonumber\\
&\quad\leq Ce^{2Re\lambda t}\{\|\pa_x \mathbf{g}_0\|_{L^2}^2+\|\mathbf{g}_0\|_{L^2}^2\}.
\end{align}
Here we have used \eqref{3.2-3.1} in the last inequality. For $v$-derivatives, we take $\pa_{v_j}$ $(j=1,2.3)$ to \eqref{3.2-0} to get
\begin{align}
&\pa_t\pa_{v_j}g_i+v_1\pa_x\pa_{v_j}g_i-L\pa_{v_j}g_i\nonumber\\
&\quad=-[\pa_{v_j},v_1\pa_x-L]g_i+\beta F(\sqrt{\mu}g_{i+1})\pa_{v_j}(v_1\sqrt{\mu})\nonumber\\
&\qquad-F(\sqrt{\mu}g_{i+1})\pa_{v_j}\pa_{v_1}g_i+\f{\beta}{2}F(\sqrt{\mu}g_{i+1})[v_1\pa_{v_j}g_i+g_i\delta_{1j}].\label{3.2-8}
\end{align}
Here the commutator satisfies $[\pa_{v_j},v_1\pa_x-L]=\delta_{1j}\pa_x+\f{\beta^2v_j}{2}$, where $\delta_{1j}=1$ for $j=1$ and $\delta_{1j}=0$ for $j=2,3$. Again, taking the inner product of \eqref{3.2-8} with $\pa_{v_j}g_i$ leads to
\begin{align}\label{3.2-9}
&\f{1}{2}\f{\dd }{\dd t}\|\pa_{v_j}g_i\|_{L^2}^2+\langle-L\pa_{v_j}g_i,\pa_{v_j}g_i\rangle\nonumber\\
&\quad=\langle-\left[\pa_{v_j},(v_1\pa_x-L)\right]g_i,\pa_{v_j}g_i\rangle+\langle\beta F(\sqrt{\mu}g_{i+1})\pa_{v_j}(v_1\sqrt{\mu}),\pa_{v_j}g_i\rangle\nonumber\\
&\qquad+\langle -F(\sqrt{\mu}g_{i+1})\pa_{v_j}\pa_{v_1}g_i,\pa_{v_j}g_i\rangle+\f\beta2\langle F(\sqrt{\mu}g_{i+1})[v_1\pa_{v_j}g_i+g_i\delta_{1j}],\pa_{v_j}g_i\rangle.
\end{align}
Similar as in \eqref{J1}, it holds that
\begin{align}
\langle -L\pa_{v_j}g_i,\pa_{v_j}g_i \rangle\geq \|\nabla_v\pa_{v_j} g_i\|_{L^2}^2+\f{\beta}{4}\|\nu^{1/2}\pa_{v_j}g_i\|_{L^2}^2-2\beta\|\pa_{v_j}g_i\|_{L^2}^2.\nonumber
\end{align}
Next to estimate the R.H.S of \eqref{3.2-9}. A direct computation shows that
\begin{align}
|\langle\left[\pa_{v_j},v_1\pa_x-L\right]g_i,\pa_{v_j}g_i\rangle|&=|\langle \delta_{1j}\pa_x{g_i}+\f{\beta^2v_j}{2}g_i,\pa_{v_j}g_i\rangle|\nonumber\\
&\leq C\|\pa_{x}g_i\|_{L^2}^2+C\|\pa_{v_j}g_i\|_{L^2}^2+C\|\nu^{1/2} g_i\|_{L^2}^2,\nonumber
\end{align}
and
$$
\langle -F(\sqrt{\mu}g_{i+1})\pa_{v_j}\pa_{v_i}g_i,\pa_{v_j}g_i\rangle=\langle -F(\sqrt{\mu}g_{i+1}),\pa_{v_i}\f{(\pa_{v_j}g_i)^2}{2}\rangle=0.
$$
By Cauchy-Schwarz, it holds that
\begin{align}
&|\langle\beta F(\sqrt{\mu}g_{i+1})\pa_{v_j}(v_1\sqrt{\mu}),\pa_{v_j}g_i\rangle|+\f\beta2|\langle F(\sqrt{\mu}g_{i+1})(g_i\delta_{1j}+v_1\pa_{v_j}g_i),\pa_{v_j}g_i\rangle|\nonumber\\
&\quad\leq C\{\|\mathbf{g}\|_{L^2}+\|\nabla_v\mathbf{g}\|_{L^2}^2\}+C\sup_{0\leq s\leq t}\|g_{i+1}(s)\|_{L^2}\cdot\|\nu^{1/2}\nabla_{v}g_i\|_{L^2}^2. \nonumber
\end{align}
This completes the estimates on the R.H.S of \eqref{3.2-9}. Integrating it over $[0,t]$, we have, for sufficiently small $\eta>0$, that
\begin{align}\label{3.2-10}
\|\nabla_v\mathbf{g}\|_{L^2}^2&\leq C\|\nabla_v\mathbf{g}_0\|_{L^2}^2+C\int_0^t\|\pa_{x}\mathbf{g}(s)\|_{L^2}^2+\|\nabla_{v}\mathbf{g}(s)\|_{L^2}^2+\|\nu^{1/2} \mathbf{g}(s)\|_{L^2}^2\dd s\nonumber\\
&\leq C\|\nabla_v\mathbf{g}_0\|_{L^2}^2+Ce^{2Re\lambda t}\{\|\mathbf{g}_0\|_{L^2}^2+\|\pa_x\mathbf{g}_0\|_{L^2}^2\}.
\end{align}
Here we have used \eqref{3.2-7} and \eqref{3.2-4} in the last inequality. Now the estimate \eqref{3.2-1} with $w\equiv1$ directly follows from \eqref{3.2-7} and \eqref{3.2-10}. Finally, we turn to the weighted estimate. Taking the inner product of \eqref{3.2-0} and \eqref{3.2-5} with $w^2g_1$ and $w^2\pa_x g_i$ respectively, adding them together and then using Cauchy-Schwarz, we obtain
\begin{align}\label{3.2-11}
&\f12\f{\dd }{\dd t}\{\|wg_i\|_{L^2}^2+\|w\pa_xg_i\|_{L^2}^2\}+\underbrace{\langle-Lg_i,w^2g_i\rangle+\langle-L\pa_xg_i,w^2\pa_xg_i\rangle}_{J_5}\nonumber\\
&\quad\leq \zeta\{\|\nu^{1/2} wg_i\|_{L^2}^2+\|\nu^{1/2} w\pa_xg_i\|_{L^2}^2\}+C_{\zeta}\|\mathbf{g}\|_{L^2}^2+C\sup_{0\leq s\leq t}\|\mathbf{g}(s)\|_{L^2}\nonumber\\
&\quad\quad\times\left\{\|\nu^{1/2} wg_i(s)\|_{L^2}^2+\|\nu^{1/2} w\pa_xg_i(s)\|_{L^2}^2+\|w\nabla_vg_i\|_{L^2}^2+\|w\pa_x\nabla_vg_i\|_{L^2}^2\right\}.
\end{align}
Here $\zeta>0$ can be chosen arbitrarily small. Similar as in \eqref{4.1-4}, it holds that
\begin{align}
J_5\geq &\|w\nabla_vg_i\|_{L^2}^2+\|w\nabla_v\pa_xg_i\|_{L^2}^2+\f{\beta^2}{16}\{\|\nu^{1/2} wg_i\|_{L^2}^2+\|\nu^{1/2} w\pa_xg_i\|_{L^2}^2\}\nonumber\\
&\quad-C\{\|g_i\|_{L^2}^2+\|\pa_xg_i\|_{L^2}^2\}.\nonumber
\end{align}
Substituting this into \eqref{3.2-11} and taking $\zeta>0$ suitably small, one has
\begin{align}\label{3.2-12}
&\f12\f{\dd }{\dd t}\{\|wg_i\|_{L^2}^2+\|w\pa_xg_i\|_{L^2}^2\}\nonumber\\
&\qquad+\|w\nabla_vg_i\|_{L^2}^2+\|w\nabla_v\pa_xg_i\|_{L^2}^2+\|\nu^{1/2} wg_i\|_{L^2}^2+\|\nu^{1/2} w\pa_xg_i\|_{L^2}^2\nonumber\\
&\quad\leq C\|\mathbf{g}\|_{L^2}^2+C\|\pa_x\mathbf{g}\|_{L^2}^2+C\sup_{0\leq s\leq t}\|\mathbf{g}(s)\|_{L^2}\nonumber\\ &\qquad\times\left\{\|\nu^{1/2} wg_i(s)\|_{L^2}^2+\|\nu^{1/2} w\pa_xg_i(s)\|_{L^2}+\|w\nabla_vg_i\|_{L^2}^2+\|w\pa_x\nabla_vg_i\|_{L^2}^2\right\}.
\end{align}
Similarly, take the inner product of \eqref{3.2-8} with $w^2\pa_{v_j}g_i$ and use Cauchy-Schwarz to obtain that
\begin{align}\label{3.2-13}
&\f12\f{\dd}{\dd t}\|w\pa_{v_j}g_i\|_{L^2}^2+\{\|\nu^{1/2} w\pa_{v_j}g_i\|_{L^2}^2+\|w\nabla_v\pa_{v_j}g_i\|_{L^2}^2\}\nonumber\\
&\quad\leq C\{\|\nu^{1/2} w\mathbf{g}\|_{L^2}^2+\|\nu^{1/2} w\pa_x\mathbf{g}\|_{L^2}^2+\| w\nabla_v\mathbf{g}\|_{L^2}^2\}\nonumber\\
&\qquad+C\sup_{0\leq s\leq t}\|\mathbf{g}(s)\|_{L^2}\cdot\{\|\nu^{1/2} wg_i\|_{L^2}^2+\|\nu^{1/2} w\nabla_v g_i\|_{L^2}^2+\| w\nabla_v^2 g_i\|_{L^2}^2\}.
\end{align}
Suitably combining \eqref{3.2-12} with \eqref{3.2-13} and taking $\eta>0$ sufficiently small, we have
\begin{align}
&\|w\mathbf{g}\|_{L^2}^2+\|w\pa_x\mathbf{g}\|_{L^2}^2+\|w\nabla_v\mathbf{g}\|_{L^2}^2\nonumber\\
&\quad\leq C\{\|w\mathbf{g}_0\|_{L^2}^2+\|w\pa_x\mathbf{g}_0\|_{L^2}^2+\|w\nabla_v\mathbf{g}_0\|_{L^2}^2\}+C\int_0^t\|\mathbf{g}(s)\|_{L^2}^2+\|\pa_x\mathbf{g}(s)\|_{L^2}^2\dd s\nonumber\\
&\quad\leq Ce^{2Re\lambda t}\{\|w\mathbf{g}_0\|_{L^2}^2+\|w\pa_x\mathbf{g}_0\|_{L^2}^2+\|w\nabla_v\mathbf{g}_0\|_{L^2}^2\}.\nonumber
\end{align}
Therefore, the proof of Lemma \ref{lm3.2} is completed.
\end{proof}
Now we are in the position to prove Theorem \ref{thm1.1}. We only sketch the proof which is similar to arguments developed in \cite{EGM1}\\

\underline{\it Proof of Theorem \ref{thm1.1}: } The proof is divided into two steps.\\

Step. 1. Initial positivity: We take $\CR$ to be the eigenvector whose eigenvalue $\lambda$ has the largest positive real part (see Lemma \ref{lm3.0}). Notice that by \eqref{3.1-1} and Sobolev embedding $H^N(\Omega\times\mathbb{R}^3)\hookrightarrow L^{\infty}(\Omega\times\mathbb{R}^3)$ ($N\geq 3$), we have the following pointwise estimate:
\begin{align}\label{R}
\sup_{x,v}|e^{\f{q\beta|v|^2}{4}}\CR(x,v)|\leq C,
\end{align}
for $0\leq q\leq 1/2$. Now we claim that there exists a sequence of approximate eigenvectors $\CR^{\vep}(x,v)$ in the sense that:
\begin{align}
&\vep |\CR^{\vep}(x,v)|\sqrt{\mu}\leq \mu,\quad\|\CR^{\vep}-\CR\|_{L^2}\leq C\vep^{1/2},\nonumber\\
&\|w\CR^{\vep}\|_{H^N_{x,v}}\leq C\|w\CR\|_{H^N_{x,v}}\leq C,\label{3.3-1}
\end{align}
for sufficiently small $\vep>0$ and some positive constants $C>0$ independent of $\vep$. Indeed, introduce the following smooth cut-off function $0\leq \chi^{\vep}(s)\leq 1$:
$$\chi^\vep(s)=\left\{\begin{aligned}
&1, \quad 0\leq s\leq \sqrt{\f{6}{\beta}|\log C\vep|},\\
&0, \quad s\geq \sqrt{\f{8}{\beta}|\log C\vep|},
\end{aligned}\right.
$$
for a suitably chosen constant $C>0$ independent of $\vep$. Let $R_i^{\vep}(x,v)=R_i(x,v)\chi^{\vep}(|v|)$. Then by the aid of \eqref{R}, we directly compute that
$$\begin{aligned}
\vep |\CR^{\vep}(x,v)|\sqrt{\mu}&\leq \vep|\CR\sqrt{\mu}\Fi_{\{|v|\leq\sqrt{\f{8}{\beta}|\log C\vep|}\}}|\leq C\vep \mu(v)|e^{\f\beta8|v|^2}\Fi_{\{|v|\leq\sqrt{\f{8}{\beta}|\log C\vep|}\}}|\\
&\leq C\vep\mu(v) e^{|\log C\vep|}\leq \mu(v),
\end{aligned}
$$
$$\begin{aligned}
\|\CR-\CR^{\vep}\|_{L^2}^2&\leq C\int_{|v|\geq \sqrt{\f{6}{\beta}|\log C\vep|}}e^{-\f\beta4|v|^2}\dd v\\
&\leq C e^{-|\log C\vep|}\cdot\int_{|v|\geq \sqrt{\f{6}{\beta}|\log C\vep|}}e^{-\f\beta{12}|v|^2}\dd v\leq C\vep,
\end{aligned}
$$
and
$$\|w\CR^{\vep}\|_{H^{N}_{x,v}}\leq C\|\chi^{\vep}\|_{C^N}\cdot\|w\CR\|_{H^N_{x,v}}\leq C.
$$
This shows \eqref{3.3-1}. Moreover, if the principle eigenvalue $\lambda_1$ is not real, one can prove, by the same argument as in Lemma 6.3 of \cite{EGM1}, that  under the dynamics of the linearized Vlasov-Fokker-Planck system, the imaginary part of its eigenvector $Im\CR$ with $\|Im\CR\|_{L^2}=r>0$ grows exponentially in the sense that
\begin{align}
\|e^{-t\CL}Im\CR\|_{L^2}\geq \varrho e^{Re\lambda_1 t}\|Im\CR\|_{L^2}\geq \varrho e^{Re\lambda_1 t}r ,\nonumber
\end{align}
for some positive constant $\varrho>0$. If $\lambda_1$ is real, we only take the real part. Define a family of initial data: $$\mathbf{f}^{\vep}(0,x,v)=\mu+\sqrt{\mu}\mathbf{g}^{\vep}(0,x,v):=\mu+\vep Im\CR^{\vep}\sqrt{\mu}.$$
Notice that from \eqref{3.3-1}, it holds that
$$\mathbf{f}^{\vep}(0,x,v)\geq 0\quad \text{and}\quad \|w\mathbf{g}^{\vep}_0\|_{H^{N}_{x,v}}\leq C_{N}\vep.
$$
\\

Step. 2. Justification of the escape time: Choose $\Lambda>0$ such that $Re\lambda_1<\Lambda<\f{3}{2}Re\lambda_1$ and define the following time points:
$$
\begin{aligned}
\hat{T}^{*}:&=\f{1}{\Lambda-Re\lambda_1}|\ln\f{\varrho r}{2C_{\Lambda}\sqrt{\vep}}|,\\
\hat{T}^{**}:&=\sup_{s}\{s:\text{ }\|\mathbf{g}^{\vep}(s)\|_{L^2}\leq \eta,\text{ for all }0\leq t\leq s \},\\
\hat{T}^{***}:&=\sup_{s}\{s:\text{ }\|\mathbf{g}^{\vep}(s)-\vep e^{-s\CL}Im\CR^{\vep}\|_{L^2}\leq \f{\varrho}{4}\vep e^{Re\lambda_1 s}r,\text{ for all }0\leq t\leq s\},
\end{aligned}
$$
and $T^{\vep}:=\f{1}{Re\lambda_1}\ln\f{\delta_0}{\vep}$. Then Theorem \ref{thm1.1} directly follows from the following claim:
\begin{align}\label{ti}
0<T^\vep\leq \min\{\hat{T}^{*},\hat{T}^{**},\hat{T}^{***}\}.
\end{align}
In fact, if \eqref{ti} holds, then one has from Lemma \ref{lm3.0} and \eqref{3.3-1} that
\begin{align}
\|\mathbf{g}^{\vep}(T^{\vep})\|_{L^2}&\geq \|\vep e^{-T^{\vep}\CL}Im\CR\|_{L^2}-\|\vep e^{-T^{\vep}\CL}(Im\CR-Im\CR^{\vep})\|_{L^2}-\|\mathbf{g}^{\vep}(T^{\vep})-\vep e^{-T^{\vep}\CL}Im\CR^{\vep}\|_{L^2}\nonumber\\
&\geq \varrho\vep e^{Re\lambda_1 T^{\vep}}r-C_{\Lambda}\vep^{3/2}e^{\Lambda T^{\vep}}-\f{\varrho}{4}\vep e^{Re\lambda_1 T^{\vep}}r\geq \f{3\varrho }{4}\vep e^{Re\lambda_1 T^{\vep}}r-C_{\Lambda}\sqrt{\delta_0}\vep e^{Re\lambda_1 T^{\vep}}\nonumber\\
&\geq \f{\varrho }{4}\vep e^{Re\lambda_1 T^{\vep}}r\geq \f{\varrho r}{4}\delta_0,\nonumber
\end{align}
provided that $\delta_0>0$ is suitably small. And so it suffices to prove \eqref{ti}. We first notice that
$$\hat{T}^{*}\geq \f{1}{Re\lambda_1}|\log\f{\varrho^2r^2}{4C^2_{\Lambda}\vep}|\geq T^{\vep},
$$
for $0<\delta_0\leq\f{\varrho^2r^2}{4C_{\Lambda}^2}.$ Next to prove $T^{\vep}\leq \min\{\hat{T}^{**},\hat{T}^{***}\}$. On one hand, if $\hat{T}^{**}\leq \min\{T^{\vep},\hat{T}^{***}\}$, then it holds that
$$\begin{aligned}
\|\mathbf{g}^{\vep}(\hat{T}^{**})\|_{L^2}&\leq\|\vep e^{-T^{\vep}\CL}Im\CR\|_{L^2}+\|\vep e^{-T^{\vep}\CL}(Im\CR-Im\CR^{\vep})\|_{L^2}+\|\mathbf{g}^{\vep}(T^{\vep})-\vep e^{-T^{\vep}\CL}Im\CR^{\vep}\|_{L^2}\\
&\leq r\vep e^{Re\lambda_1 T^{\vep}}+C_{\Lambda}\vep^{3/2}e^{\Lambda T^{\vep}}+\f{\varrho}{4}\vep e^{Re\lambda_1 T^{\vep}}r\leq C\delta_0.
\end{aligned}
$$
This leads to a contradiction to the definition of $\hat{T}^{**}$, if we choose $0<\delta_0\leq \f{\eta}{2C}.$ On the other hand, assume that $\hat{T}^{***}\leq \min\{T^{\vep},\hat{T}^{**}\}$. Then for any $0\leq t\leq \hat{T}^{***}$, it holds that
\begin{align}\label{e}
\|\mathbf{g}^\vep(t)\|_{L^2}&\leq \|\vep e^{-\CL t}Im\CR\|_{L^2}+\|\vep e^{-\CL t}(Im\CR-Im\CR^{\vep})\|_{L^2}+\|\mathbf{g}^{\vep}(t)-\vep e^{-\CL t}Im\CR^{\vep}\|_{L^2}\nonumber\\
&\leq \vep e^{Re\lambda_1 t}r+C_{\Lambda}\vep^{3/2}e^{\Lambda t}+\f{\varrho}{4}\vep e^{Re\lambda_1 t}r\leq C\vep e^{Re\lambda_1 t}r.
\end{align}
Recall the equation \eqref{3.2-0}. Denote the nonlinear term on the R.H.S as $\mathbf{N}(\mathbf{g})=[N_1(\mathbf{g}),N_2(\mathbf{g})]$ where $$N_i=-F(\sqrt{\mu}g_{i+1})\pa_{v_1}g_i+\f{\beta}{2}F(\sqrt{\mu}g_{i+1})v_1g_{i}.$$
Then applying the Duhamel principle to \eqref{3.2-0} and using Lemma \ref{lm3.0}, Lemma \ref{lm3.2} and \eqref{e}, we have
$$
\begin{aligned}
\|\mathbf{g}^{\vep}(t)-\vep e^{-\CL t}Im\CR^{\vep}\|_{L^2}&=\|\int_{0}^te^{-\CL(t-s)}\mathbf{N}(\mathbf{g}^\vep)(s)\dd s\|_{L^2}\\
&\leq C_{\Lambda}\int_{0}^te^{\Lambda(t-s)}e^{2Re\lambda_1 s}\{\|w\pa_{x,v}\mathbf{g}^{\vep}_0\|_{L^2}+\|w\mathbf{g}^{\vep}_0\|_{L^2}\}^2\leq C\vep^2e^{2Re\lambda_1 t}.
\end{aligned}
$$
And so, at $t=\hat{T}^{***}$, one has
$$\|\mathbf{g}^{\vep}(\hat{T}^{***})-\vep e^{-\CL \hat{T}^{***}}Im\CR^{\vep}\|_{L^2}\leq C\vep^2e^{Re\lambda_1 T^{\vep}}\cdot e^{Re\lambda_1\hat{T}^{***}}\leq C\delta_0\vep e^{Re\lambda_1 \hat{T}^{***}}.
$$
This is a contradiction to the definition of $\hat{T}^{***}$ if we take $\delta_0$ suitably small. Therefore, \eqref{ti} holds and the proof of Theorem \ref{thm1.1} is completed. $\hfill\Box$\\

\noindent{\bf Acknowledgments.} This work was done when the author was visiting Institute for Analysis and Scientific Computing at Vienna University of Technology. He thanks them for their kind hospitality and gratefully acknowledges support from the Eurasia-Pacific Uninet Ernst Mach Grant.

\end{document}